\documentclass[11pt,english]{svjour3}
\usepackage[T1]{fontenc}
\usepackage[latin9]{inputenc}
\usepackage{geometry}
\usepackage{url}
\usepackage{amsmath}
\usepackage{amssymb}
\usepackage{setspace}
\usepackage{esint}
\makeatletter
\numberwithin{equation}{section}
\numberwithin{figure}{section}

\makeatother

\usepackage{babel}

\begin{document}
\global\long\def\Re{\mathfrak{Re}}

\global\long\def\Im{\mathfrak{Im}}

\global\long\def\One{{\bf 1 }}

\global\long\def\ZZ{\mathbb{Z}}

\global\long\def\NN{\mathbb{N}}

\global\long\def\RR{\mathbb{R}}

\global\long\def\PP{\mathbb{P}}

\global\long\def\BB{\mathcal{B}}

\global\long\def\O{\mathrm{\mathtt{\mathsf{\mathbf{\mathrm{O}}}}}}

\global\long\def\o{\mathrm{\mathtt{\mathrm{o}}}}

\global\long\def\DD{\mathcal{D}}

\global\long\def\deq{\overset{d}{=}}

\title{Weak invariance principle for the local times of partial sums of
Markov Chains}

\author{Michael Bromberg and Zemer Kosloff}

\institute{School of Mathematical Sciences, Tel Aviv University. Tel Aviv 69978,
Israel.}
\maketitle
\begin{abstract}
Let $\left\{ X_{n}\right\} $ be an integer valued Markov Chain with
finite state space. Let $S_{n}=\sum_{k=0}^{n}X_{k}$ and let $L_{n}\left(x\right)$
be the number of times $S_{k}$ hits $x\in\ZZ$ up to step $n.$ Define
the normalized local time process $l_{n}\left(t,x\right)$ by 
\[
l_{n}\left(t,x\right)=\frac{L_{\left\lfloor nt\right\rfloor }\left(\left\lfloor \sqrt{n}x\right\rfloor \right)}{\sqrt{n}},\,\, x\in\RR.
\]
 The subject of this paper is to prove a functional, weak invariance
principle for the normalized sequence $l{}_{n}\left(t,x\right),$i.e.
we prove under the assumption of strong aperiodicity of the Markov
Chain that the normalized local times converge in distribution to
the local time of the Brownian Motion. 
\end{abstract}

\section{\label{sub:General-Definitions}Introduction.}

Let $S\subseteq\ZZ$ be a finite set, $P:S\times S\rightarrow[0,1]$
an irreducible, aperiodic stochastic matrix. Let $\mu$ be some distribution
on $S$ and let $\left\{ X_{i}\right\} _{i=0}^{\infty}$ be a Markov
Chain generated by $P$ and the initial distribution $\mu$, i.e.
$X_{i},\, i=0,1,2,...$ are random variables defined on a probability
space $\left(X,\BB,P^{\mu}\right)$, taking values in $S$, such that
the equality

\[
\PP^{\mu}\left(X_{i_{1}}=s_{1},\, X_{i_{2}}=s_{2},...,\, X_{i_{n}}=s_{n}\right)=\sum_{s_{0}\in S}\mu_{s_{0}}\cdot P^{i_{1}}\left(s_{0},s_{1}\right)\cdot P^{i_{2}-i_{1}}\left(s_{1},s_{2}\right)...\cdot P^{i_{n}-i_{n-1}}\left(s_{n-1},s_{n}\right),
\]
holds with $\mu_{s}=\mu\left\{ s\right\} .$ Under these assumptions,
$P$ has a unique $P$-invariant probability distribution vector $\nu$,
in the sense that the equality $\nu P=\nu$ holds. 

We may assume that for every initial distribution $\mu$, $X=S^{\NN},$
$\BB$ is the Borel $\sigma$-field generated by cylinders of the
type $\left[s_{0}s_{1}...s_{n}\right]=\left\{ \omega\in S^{\NN}:\,\omega_{i}=s_{i},\, i=0,1,...,n\right\} $
and $X_{i}\left(\omega\right)=\omega_{i}.$ 

Set 
\[
S_{n}=\sum\limits _{k=0}^{n}X_{k},\ W_{n}(t)=\frac{1}{\sqrt{n}}S_{\left\lfloor nt\right\rfloor },t\in[0,1].
\]

Let $D_{[s,t]}$, $D_{+}$ and $D$ denote the spaces of functions
defined on $\left[s,t\right]$, $\left[0,\infty\right)$ and $\RR$
respectively, that are continuous on the right with left limits at
each point (Cadlag functions). $D_{\left[s,t\right]}$, $D_{+}$ and
$D$ endowed with the Skorokhod $J_{1}$ topology and its suitable
generalization to $\left[0,\infty\right)$ and $\RR$, constitute
Polish (complete and separable metrizable) spaces (see \cite{Bil}). 

Under the assumption that $E^{\nu}\left(X_{0}\right)=0$, the weak
invariance principle for the sequence $W_{n}$ holds for any initial
distribution $\mu,$ i.e. the process $W_{n}\left(t\right)$ regarded
as a sequence of random variables defined on $(X,\BB,P^{\mu})$ and
taking values in $D_{+}$ converges in distribution to the Brownian
motion, henceforth denoted by $W_{\sigma}\left(t\right),$ where 
\[
\sigma^{2}=Var\left(W_{\sigma}\left(1\right)\right)=\lim\limits _{n\rightarrow\infty}\frac{E^{\nu}\left(S_{n}^{2}\right)}{n}
\]
 is the asymptotic variance of the process $S_{n}$ with respect to
the initial distribution $\nu.$ Now, $\sigma^{2}=0$ if and only
if $X$ is a coboundary, meaning that there exists a square integrable
process $Y=\left\{ Y_{n}\right\} $ such that (see \cite[Lemma 6]{R-E}
or \cite[II.3 Thm.A]{HeH}) 
\[
X_{n}=Y_{n}-Y_{n+1}.
\]

In this case the functional limit of $W_{n}$ is degenerate, so we
restrict our attention to the case when $\sigma^{2}>0$. The invariance
principle may be proved by techniques used in this article, or otherwise,
see \cite{Bil}.

Let $L_{n}(x)=\#\left\{ k\leq n:S_{k}=x\right\} =\sum_{k=0}^{n}\One_{\left\{ x\right\} }\left(S_{k}\right)$
denote the number of arrivals of the process $\left\{ S_{k}\right\} _{k\in\NN}$
at the point $x$ after $n$ steps, and let $l_{n}(t,x)=\frac{1}{\sqrt{n}}L_{\left[nt\right]}\left(\left[\sqrt{n}x\right]\right)$,
$\left(t,x\right)\in\left[0,\infty\right]\times\RR$. $l_{n}\left(t,x\right)$
is the amount of time multiplied by $\sqrt{n}$ that the process $W_{n}\left(\cdot\right)$
spends up to the time instant $t$ at the point $\frac{\left[\sqrt{n}x\right]}{\sqrt{n}}$
. 

Let $l\left(t,x\right)$ be the local time of the Brownian motion
$W_{\sigma}\left(\cdot\right)$ at time $t$, i.e. $l\left(\cdot,\cdot\right)$
is a random function satisfying the equality 
\[
\int\limits _{A}l\left(t,x\right)dx=\int_{0}^{t}\One_{A}\left(W_{\sigma}\left(t\right)\right)dt
\]
 for every Borel set $A\in\BB_{\RR}$ and $t\in\left[0,\infty\right]$,
$\PP$ - a.s where $\PP$ is the Wiener measure on $C(\RR)$, the
space of continuous functions on $\RR$ (see \cite{MP}). 

Note, that for every $t\geq0$, the sequence $l_{n}\left(t,x\right)$
defines a $D$ valued process on $X$. Since, $l(t,\cdot)$ is a.s.
continuous on $\RR$, we may also regard it as a $D$ valued random
variable.

In this paper, we prove that under the assumption that the Markov
chain is strongly aperiodic (see definition \ref{def:Aperiodicity})
, the weak invariance principle for $W_{n}\left(t\right)$ implies
an invariance principle for the local times, i.e. the convergence
in distribution of $W_{n}$ to $W_{\sigma}$ implies that $l_{n}\left(t,x\right)$
converges to $l\left(t,x\right)$ (here, convergence is in the sense
that there exists a probability space where $l_{n}\left(t,x\right)$
converges to $l\left(t,x\right)$ uniformly on compact sets). The
case when $X_{i}$'s are i.i.d's was treated by Borodin \cite{Bor},
under the assumption that for every $t\notin2\pi\ZZ$, $Ee^{itX_{1}}\neq1$,
an assumption which we refer to as non-arithmeticity. We also show
that the assumption of strong aperiodicity may be exchanged for the
weaker assumption of non-arithmeticity in the i.i.d case, and for
a certain class of Markov chains (see the remark following Definition
\ref{def:Aperiodicity}).

\subsubsection*{Conventions and Notations: }
\begin{itemize}
\item $\PP^{x}$ denotes the measure $\PP^{\mu}$, where $\mu$ assigns
mass $1$ to the point $x\in S$. 
\item $E^{\mu}$ denotes integrals on $S$ with respect to the measure $\PP^{\mu}$
and by $E^{x}$ integrals with respect to the measure $\PP^{x}$. 
\item For two random variables $X$ and $Y$, $X\deq Y$ means that $X$
and $Y$ are equally distributed. Also $X_{n}\overset{d}{\Rightarrow}X$
means that $X_{n}$ converges in distribution to $X$. 
\item $\mathbb{C}^{S}$ denotes the space of complex valued functions on
$S$, which is isomorphic to $\mathbb{C}^{n}$ with $n=\left|S\right|$. 
\item $\mathcal{L_{S}}$ denotes the space of linear operators from $\mathbb{C}^{S}$
into itself, and we denote by $\left\Vert \cdot\right\Vert $ the
operator norm on this space. 
\item Throughout the paper we assume that $P$ is aperiodic and irreducible,
$\nu$ always denotes the stationary distribution with respect to
$P$. 
\end{itemize}

\section{Statement of the Main Theorem.}

Before stating the main theorem we introduce the notion of strong
aperiodicity (see \cite{GH}):
\begin{definition}
\textit{\label{def:Aperiodicity}A function $f:S\times S\rightarrow\mathbb{Z}$
is said to be aperiodic if the only solutions for }

\textit{
\[
e^{itf\left(x,y\right)}=\lambda\frac{\varphi\left(x\right)}{\varphi\left(y\right)},\quad\forall x,y\in S,\; such\: that\:\mathbb{P}^{x}\left(y\right)>0
\]
}

\textit{with $t\in\RR$, $\left|\lambda\right|=1$, $\varphi\in\mathbb{C}^{S}$,
$\left|\varphi\right|\equiv1$ are $t\in2\pi\ZZ,$ $\varphi\equiv const$.
$f$ is is said to be periodic if it is not aperiodic.}

\textit{Since, we are only interested in the case $f\left(x,y\right)=y$,
which gives $f\left(X_{n-1},X_{n}\right)=X_{n}$, we say that the
Markov Chain $\left\{ X_{n}\right\} $ is strongly aperiodic iff the
function $f\left(x,y\right)=y$ is aperiodic. }
\end{definition}
The assumption of strong aperiodicity may be dropped in case $\left\{ X_{i}\right\} $
is a sequence of i.i.d's and may be dropped in the Markov case provided
that the underlying Markov shift is almost onto (see section \ref{sec:The-Periodic-Case}).
Note that strong aperiodicity of the Markov Chain implies aperiodicity
of the stochastic matrix $P$, while the reverse implication is generally
false - a simple random walk where $X_{n}$ equals $\pm1$ with equal
probability may serve as a counter example. Here for $t=\pi,\lambda=-1$
and $\varphi\equiv1$, 
\[
e^{i\pi X_{n}}=-1,
\]
hence $X_{n}$ is not strongly aperiodic. However 
\[
P=\left(\begin{array}{cc}
\frac{1}{2} & \frac{1}{2}\\
\frac{1}{2} & \frac{1}{2}
\end{array}\right)
\]
is aperiodic. For discussion of condition \ref{def:Aperiodicity}
in a more general setting, see \cite{AD}. 

\begin{theorem}
\label{thm: Borodin for MC}Let $\left\{ X_{n}\right\} $ be an irreducible,
strongly aperiodic, finite state Markov chain. Assume that $E^{\nu}\left(X_{0}\right)=0$\textup{
}and 
\[
W_{n}\overset{d}{\Rightarrow}W_{\sigma},\ \sigma>0.
\]
Then there exists a probability space $\left(X',\BB',\PP'\right)$
and processes $W'_{n},\, W'_{\sigma}:X\rightarrow D_{+}$ such that:\end{theorem}
\begin{enumerate}
\item \textit{$W_{n}\overset{d}{=}W'_{n};\, W_{\sigma}\overset{d}{=}W'_{\sigma}$.}
\item \textit{With probability one $W'_{n}$ converges to $W'_{\sigma}$
uniformly on compact subsets of $\left[0,\infty\right)$.}
\item \textit{For every $\epsilon,T>0$ the processes $l'_{n}\left(t,x\right)$
and $l'\left(t,x\right)$ defined with respect to $W'_{n}$ and $W'_{\sigma}$
satisfy the relationship:
\begin{equation}
\lim_{n\rightarrow\infty}\PP'\left(\sup_{t\in\left[0,T\right],\, x\in\RR}\left|l'_{n}\left(t,x\right)-l'\left(t,x\right)\right|>\epsilon\right)=0.\label{eq: Uniform convergence on compacts for local time}
\end{equation}
}
\end{enumerate}
We start by fixing the time variable at $1$, and for convenience
denote: $t_{n}\left(x\right)\equiv l_{n}\left(1,x\right)$ and $t\left(x\right)\equiv l\left(1,x\right).$
The main step on the road to proving the previous theorem is to prove:
\begin{theorem}
\label{thm: Main Theorem}Let $\left\{ X_{n}\right\} $ be a finite
state Markov chain, $E^{\nu}\left(X_{0}\right)=0$ and 
\[
W_{n}\overset{d}{\Rightarrow}W_{\sigma},\ \sigma>0,
\]
then 
\[
t_{n}\overset{d}{\Rightarrow}t_{\sigma}.
\]

\end{theorem}
In order to prove convergence in distribution of $\left\{ t_{n}(x)\right\} $,
we have to establish relative compactness of the family $\Pi=\left\{ t_{n}\left(x\right)\right\} _{n\in\NN},$
which by definition means that every sequence of elements from $\Pi$
has a subsequence that converges in distribution. This is done in
section \ref{sub: proof of the main theorem}. Then to complete the
proof, in section \ref{sub:Identifying the Only Possible Limit Point}
we show that $t_{\sigma}\left(x\right)$ is the only possible distributional
limit point.

\subsubsection*{Organization of the paper: }

We begin with the proof under the assumption that $X_{n}$ is strongly
aperiodic. Section \ref{sub:The Characteristic Function Operator}
contains estimates of the characteristic functions of the relevant
processes, which are later used, in section \ref{sub:Probabilistic Estimates}
to carry out the calculations needed for the proof of Theorem \ref{thm: Main Theorem}.
In sections \ref{sub: proof of the main theorem} and \ref{sec:Borodin's-Theorem.}
we prove Theorems \ref{thm: Main Theorem} and \ref{thm: Borodin for MC}
respectively under the assumption of strong aperiodicity.

In Section \ref{sec:The-Periodic-Case} we show the modifications
needed to extend these results to the periodic case. 

This work was motivated by Aaronson's work on \textit{random walks
in random sceneries} with independent jump random variables. In section
\ref{sec:Applications-to-complexity} we explain the model of RWRS
and why Theorem \ref{thm: Borodin for MC} implies Aaronson's result
for RWRS with the distribution of a strongly aperiodic Markov Chain.

\section{\label{sub:The Characteristic Function Operator}The Characteristic
Function Operator $Q(t)$}

From the irreducibility of $P$ it follows that there is a unique
probability distribution $\nu$ on $S$, such that $\nu P=\nu$. This
is equivalent to $1$ being a simple eigenvalue of $P$, with a right-hand
eigenspace of vectors in $\mathbb{C}^{S}$ having equal entries. Moreover,
from aperiodicity of $P$ it follows that all other eigenvalues of
$P$ are of modulus strictly less than $1.$ The projection to the
eigenspace belonging to $1$ is given by $\left\langle v,\cdot\right\rangle \One=E^{\nu}\left(\cdot\right)\One$
(here, $\One$ is a vector in $\mathbb{C}^{S}$ with all entries equal
to $1$). In what follows, we omit the $\One$, so that depending
on the context $E^{\nu}\left(f\right)$ may denote both a scalar,
or a vector in $\mathbb{C}^{S}$ with all entries equal to $E^{\nu}\left(f\right)$.
Hence, we can write

\[
Pf=E^{\nu}\left(f\right)+Nf,
\]
 where $N$ is an operator with spectral radius $\rho\left(N\right)$
strictly less then $1$. It follows that for every $\rho\left(N\right)<\eta<1$,
and $n$ large enough, we have 
\[
\left\Vert N^{n}\right\Vert \leq\eta^{n}.
\]

Let $Q\left(t\right):\mathbb{C}\rightarrow\mathcal{L_{S}}$ be an
operator valued function defined by 
\[
Q\left(t\right)f\left(x\right):=\int\limits _{S}e^{ity}f\left(y\right)\PP^{x}\left(dy\right)=E^{x}\left(e^{itX_{0}}f\right)=\sum_{y\in S}P\left(x,y\right)e^{ity}f\left(y\right).
\]
 Note that $Q\left(0\right)=P$. 

Since, $S$ is a finite state space, it is easy to see that the function
$Q\left(t\right)$ is holomorphic and its derivatives are given by 

\[
\left(\frac{d^{n}}{dt^{n}}Q\left(t\right)f\right)\left(x\right)=i^{n}\int y^{n}e^{ity}f\left(y\right)\,\PP^{x}\left(dy\right)=i^{n}\sum_{y\in S}y^{n}P\left(x,y\right)e^{ity}f\left(y\right).
\]

Let 
\[
\varphi_{n}^{\mu}\left(t\right)=E^{\mu}\left(e^{itS_{n}}\right),\ \varphi_{n}^{\mu}\left(t_{1},t_{2}\right)=E^{\mu}\left(e^{it_{1}S_{n}+it_{2}X_{n}}\right)
\]
 denote the characteristic functions of the processes $S_{n}$ and
$\left(S_{n},\, X_{n}\right)$ respectively, under the initial distribution
$\mu.$ Our primary interest in $Q\left(t\right)$ is due to the fact
that $\varphi_{n}^{\mu}\left(t\right)=E^{\mu}\left(Q\left(t\right)^{n}\One\right),$
and $\varphi_{n}^{\mu}\left(t_{1},t_{2}\right)=E^{\mu}\left(Q\left(t_{1}\right)^{n}e^{it_{2}s}\right).$

\subsection{Expansion of the largest eigenvalue.}

It may be shown using standard perturbation techniques (see \cite{Ka}
or \cite[ Chapter 3]{HeH}) that the largest eigenvalue of $Q\left(t\right)$
and its eigenspaces are analytic functions of $t$ in a neighborhood
of $0.$ More precisely, we have that in a neighborhood $I$ of $0$,
$Q$$\left(t\right)$ has a simple eigenvalue $\lambda\left(t\right)$
which is an analytic perturbation of $\lambda\left(0\right)$, and
there is a positive gap between $\lambda\left(t\right)$ and all other
eigenvalues of $Q\left(t\right)$, i.e. 
\[
\inf\limits _{t\in I}\min\limits _{\tilde{\lambda}\left(t\right)}\left|\lambda\left(t\right)-\tilde{\lambda}\left(t\right)\right|\geq const>0,
\]
where the maximum is taken over all remaining eigenvalues of $Q\left(t\right)$.
The projection function to the corresponding eigenspaces $\Pi\left(t\right)$
is also analytic in $t$, and we may choose the eigenfunctions $v\left(t\right)$
to be analytic perturbations of $v\left(0\right)=\One.$ 

Continuity of $Q\left(t\right)$ immediately gives us that in a neighborhood
$I$ of $0$ 

\begin{equation}
Q\left(t\right)=\lambda\left(t\right)\cdot\Pi\left(t\right)+N\left(t\right),\label{eq:Q(t) structure}
\end{equation}
where $\sup_{t\in I}\left\Vert N\left(t\right)\right\Vert =1-\eta<1$.

We proceed by evaluating the functions $\lambda\left(t\right)$ and
$\Pi\left(t\right)$ for small $t$. By the structure of $P=Q\left(0\right)$
discussed in the beginning of this section, we have 

\begin{eqnarray}
\lambda\left(0\right) & = & 1,\nonumber \\
v\left(0\right) & = & 1,\label{eq: Q for small t}\\
\Pi(0)(\cdot) & = & \left\langle \nu,\cdot\right\rangle v\left(0\right)=E^{\nu}(\cdot).\nonumber 
\end{eqnarray}

\begin{lemma}
\label{lem: Expansion of Main Eigenvalue Lemma}If $E^{\nu}\left(X_{0}\right)=0$,
then there exists a real nonnegative constant $\sigma^{2}$ such that,
\[
\lambda\left(t\right)=1-\frac{\sigma^{2}}{2}t^{2}+\O\left(\left|t\right|^{3}\right),
\]
and 
\[
\Pi\left(t\right)\left(\cdot\right)=E^{\nu}\left(\cdot\right)+\O\left(\left|t\right|\right).
\]
\end{lemma}
\begin{proof}
The statement concerning $\Pi\left(t\right)$ immediately follows
from Taylor's expansion and formula \eqref{eq: Q for small t}, while
the expansion of the main eigenvalue is carried out in \cite[Lemma IV.4']{HeH}. \end{proof}
\begin{remark}
It may be shown \cite[Lemma IV.3, pp 24-25]{HeH} that $\sigma^{2}$
actually equals the asymptotic variance ${\displaystyle \lim_{n\to\infty}\frac{E^{\nu}\left(S_{n}^{2}\right)}{n}}$
and therefore, $\sigma^{2}=0$ corresponds to the case of a degenerate
limit for $W_{n}$. Since this case is excluded in the main theorem,
in what follows, we assume that the results of the previous lemma
hold with $\sigma^{2}>0$. 
\end{remark}

\subsection{The connection between strong aperiodicity and the spectrum of $Q(t)$.}

The next lemma connects the notion of strong aperiodicity with the
spectrum of the characteristic function operator:
\begin{lemma}
\label{lem:Aperiodicity} An aperiodic and irreducible Markov chain
$\left\{ X_{n}\right\} $ is strongly aperiodic iff $\rho\left(Q\left(t\right)\right)<1$
for all real $t\neq2\pi\ZZ$ ($\rho\left(Q\left(t\right)\right)$
is the spectral radius of $Q\left(t\right)$).\end{lemma}
\begin{proof}
Note that since $Q\left(t\right)$ is a finite dimensional operator
of norm less than or equal to $1$, $\rho\left(Q\left(t\right)\right)<1$
is equivalent to demanding that $Q\left(t\right)$ has no eigenvalues
of modulus $1$. 

If $\left\{ X_{n}\right\} $ is not strongly aperiodic, we have 
\[
e^{ity}=\lambda\frac{\varphi\left(x\right)}{\varphi\left(y\right)},\quad\forall x,y\in S\: such\: that\: P^{x}\left(y\right)>0
\]
where $\left|\lambda\right|=1,$$\left|\varphi\right|\equiv1,$ $t\notin2\pi\mbox{\ensuremath{\ZZ}}$
or $t\in2\pi\mathbb{Z}$ and $\varphi\neq const$. In the first case
notice that, 
\begin{equation}
\left(Q\left(t\right)\varphi\right)\left(x\right)=\int\limits _{S}e^{ity}\varphi\left(y\right)\PP^{x}\left(dy\right)=\lambda\varphi\left(x\right)\label{eq: eigenvalue of Q(t)}
\end{equation}
whence $\rho\left(Q\left(t\right)\right)=1$ for $t\notin2\pi\mathbb{Z}.$

In the second case ( $t\in2\pi\mathbb{Z})$, since $y\in\mathbb{Z}$
we have 
\[
\left(Q(t)\varphi\right)(x)=\left(Q(0)\varphi\right)(x)=\int_{S}\varphi(y)P^{x}(dy)=\lambda\varphi(x).
\]
This is impossible for non constant $\varphi$ since we have a spectral
gap property for $t=0$ (see the beginning of the section). This proves
one direction of the iff statement. 

To prove the opposite direction observe that if equation \eqref{eq: eigenvalue of Q(t)}holds
with $\left|\lambda\right|=1$ we have

\[
\left|\varphi\left(x\right)\right|=\left|\left(Q\left(t\right)\varphi\right)\left(x\right)\right|\leq\int\limits _{S}\left|\varphi\left(y\right)\right|\PP^{x}\left(dy\right)=\left(P\left|\varphi\right|\right)\left(x\right).
\]
Since, $1$ is a simple eigenvalue of $P$, we must have $\left|\varphi\right|\equiv const$
and therefore, we may assume that $\left|\varphi\right|\equiv1$.
Hence since 
\[
\int\limits _{S}e^{ity}\varphi\left(y\right)\PP^{x}\left(dy\right)=\lambda\varphi\left(x\right),
\]
and $\left|\lambda\varphi(x)\right|=\left|\varphi(y)\right|=1$, it
follows from Proposition \ref{Trivial proposition} in the Appendix
that 
\[
e^{ity}\varphi\left(y\right)=\lambda\varphi\left(x\right),\quad\forall x,y\in S,\: with\:\PP^{x}\left(y\right)>0
\]
 and the claim follows.
\end{proof}

\section{\label{sub:Probabilistic Estimates}Estimates.}

In this section we derive the main probability estimates needed to
prove the main theorem. When lemma \ref{lem: Expansion of Main Eigenvalue Lemma}
is used, it is always assumed to hold with $\xi>0$ (see the remark
that follows the lemma). 
\begin{lemma}
\label{lem:Non-Arit. Lemma}If $\{X_{n}\}$ is strongly aperiodic,
then for every $\delta>0,$ there are $N\in\NN,$ $0<r_{\delta}<1,$
$c>0$ such that for every $n>N,$ $t\in\left[-\pi,\pi\right]\setminus\left(-\delta,\delta\right),$
\[
\left\Vert Q\left(t\right)^{n}\right\Vert \leq cr_{\delta}^{n}.
\]
 \end{lemma}
\begin{proof}
Fix $\delta>0$ and $t\in[-\pi,\pi]\backslash(-\delta,\delta)$. By
strong aperiodicity (Lemma \ref{lem:Aperiodicity}), $r_{t}:=r\left(Q\left(t\right)\right)<1.$ 

Fix, $r_{t}<\tilde{r}_{t}<1$ and choose $n_{0}$ such that 
\[
r_{t}=\inf_{n\in\NN}\left\Vert Q(t)^{n}\right\Vert ^{\frac{1}{n}}\leq\left\Vert Q(t)^{n_{0}}\right\Vert ^{\frac{1}{n_{0}}}<\tilde{r}_{t}.
\]
 By continuity of $Q\left(\cdot\right)$ at $t$, there exists $\eta=\eta_{t}>0$,
such that if $\tau\in(t-\eta,t+\eta)$, 
\begin{equation}
\left\Vert Q\left(\tau\right)^{n_{0}}\right\Vert <\tilde{r}_{t}^{n_{0}}.\label{eq: spectral bound on a neighborhood}
\end{equation}
For $n\in\mathbb{N},$ write $n=n_{0}m+k$ where $k<n_{0}$. It follows
from \eqref{eq: spectral bound on a neighborhood} that for every
$\tau\in\left(t-\eta,t+\eta\right)$, 
\[
\left\Vert Q\left(\tau\right)^{n}\right\Vert \leq\left\Vert Q\left(\tau\right)^{n_{0}}\right\Vert ^{m}\cdot\left\Vert Q\left(\tau\right)^{k}\right\Vert \leq\tilde{r_{t}}^{n_{0}m}\cdot\left\Vert Q\left(\tau\right)^{k}\right\Vert \leq\tilde{r}_{t}^{n}\cdot\frac{\left\Vert Q\left(\tau\right)\right\Vert ^{k}}{\tilde{r}_{t}^{k}}.
\]
Setting $c_{t}=\sup\frac{\left\Vert Q\left(x\right)\right\Vert ^{k}}{\tilde{r}_{t}^{k}}$where
the supremum is taken over all $x\in\left(t-\eta,t+\eta\right),\, k\leq n_{0}$
we obtain 
\[
\left\Vert Q\left(\tau\right)^{n}\right\Vert \leq c_{t}\tilde{r}_{t}^{n}.
\]

Hence, for every $t\in\left[-\pi,\pi\right]\setminus\left(-\delta,\delta\right),$
we can pick positive numbers $n_{t},$ $\eta_{t},c_{t}$ and $r_{t}<1$
such that for every $n>n_{t}$ and $\tau\in\left(t-\eta_{t},t+\eta_{t}\right)$,
\[
\left\Vert Q\left(\tau\right)^{n}\right\Vert \leq c_{t}r_{t}^{n}.
\]
Since, $\left[-\pi,\pi\right]\setminus\left(-\delta,\delta\right)$
is compact there exists a finite sequence $\left\{ t_{i}\right\} _{i=1,..l}$
such that 
\[
\left[-\pi,\pi\right]\setminus\left(-\delta,\delta\right)\subset\bigcup\limits _{i=1}^{l}\left(t_{i}-\eta_{t_{i}},t_{i}+\eta_{t_{i}}\right).
\]

The result follows by setting $N=\max\left\{ n_{t_{i}}\right\} ,$
$r_{\delta}=\max\left\{ r_{t_{i}}\right\} ,$ $c=\max\left\{ c_{t_{i}}\right\} $. \end{proof}
\begin{lemma}
\label{lem:LLT}If $\left\{ X_{n}\right\} $ is strongly aperiodic
and $E^{\nu}\left(X_{0}\right)=0,$ there exists a constant $C$ such
that for any initial distribution $\mu,$ $n\in\NN,$ $x\in\ZZ,$
\[
P^{\mu}\left(S_{n}=x\right)\leq\frac{C}{\sqrt{n}}.
\]
 \end{lemma}
\begin{proof}
In view of equation \eqref{eq:Q(t) structure} and lemma \ref{lem: Expansion of Main Eigenvalue Lemma},
there exist $\delta,\, a,\,\eta>0$ such that for $t\in(-\delta,\delta)$,
\[
Q\left(t\right)=\lambda\left(t\right)\Pi\left(t\right)+N\left(t\right)
\]

where 
\[
\left|\lambda\left(t\right)\right|\leq1-at^{2},\quad\left\Vert N\left(t\right)^{n}\right\Vert \leq\left(1-\eta\right)^{n}.
\]

By the inversion formula for Fourier transform, 

\begin{eqnarray}
\sqrt{n}\PP^{\mu}\left(S_{n}=x\right) & = & \frac{\sqrt{n}}{2\pi}\int\limits _{-\pi}^{\pi}E^{\mu}\left(Q\left(t\right)^{n}\One\right)e^{-itx}dt\\
 & \leq & \left|\frac{\sqrt{n}}{2\pi}\int\limits _{-\delta}^{\delta}E^{\mu}\left(Q\left(t\right)^{n}\One\right)e^{-itx}dt\right|+\frac{\sqrt{n}}{2\pi}\int\limits _{R_{\delta}}\left\Vert Q\left(t\right)^{n}\right\Vert dt\label{eq:LLT Bound 1}
\end{eqnarray}

By lemma \ref{lem:Non-Arit. Lemma}, there are $N\in\NN,$ $0<r_{\delta}<1,$
$c>0$ such that for any $n>N$, 
\[
\frac{\sqrt{n}}{2\pi}\int\limits _{R_{\delta}}\left\Vert Q\left(t\right)^{n}\right\Vert dt\leq c\sqrt{n}r_{\delta}^{n}.
\]
Therefore, the second term in \eqref{eq:LLT Bound 1} is uniformly
bounded in $n$ . We proceed with estimating the first term in \eqref{eq:LLT Bound 1}.
\begin{eqnarray}
\left|\frac{\sqrt{n}}{2\pi}\int\limits _{-\delta}^{\delta}E^{\mu}\left(Q\left(t\right)^{n}\One\right)e^{-itx}dt\right| & \leq & \frac{\sqrt{n}}{2\pi}\int\limits _{-\delta}^{\delta}\left|\lambda\left(t\right)\right|^{n}dt+\frac{\sqrt{n}}{2\pi}\int\limits _{-\delta}^{\delta}\left\Vert N\left(t\right)^{n}\right\Vert dt\\
 & \leq & \frac{\sqrt{n}}{2\pi}\int\limits _{-\delta}^{\delta}\left(1-at^{2}\right)^{n}dt+\frac{2\delta\sqrt{n}}{2\pi}\left(1-\eta\right)^{n}.\label{eq: LLT Bound 2}
\end{eqnarray}

It is obvious that the second term in \eqref{eq: LLT Bound 2} tends
to $0$. To see that the first term is bounded, a change of variables
$t=\frac{x}{\sqrt{n}}$ gives us, 

\begin{eqnarray*}
\frac{\sqrt{n}}{2\pi}\int\limits _{-\delta}^{\delta}\left(1-at^{2}\right)^{n}dt & = & \frac{1}{2\pi}\int\limits _{-\delta\sqrt{n}}^{\delta\sqrt{n}}\left(1-\frac{ax^{2}}{n}\right)^{n}dx\\
 & \leq & \frac{1}{2\pi}\int\limits _{-\delta\sqrt{n}}^{\delta\sqrt{n}}e^{-ax^{2}}dx.
\end{eqnarray*}
Since, the last integral is uniformly bounded in $n$, this completes
the proof.\end{proof}
\begin{lemma}
(Estimation of the Potential Kernel, see \cite{JP})\label{pro:Potential Kernel Estimate}
If $\left\{ X_{n}\right\} $ is strongly aperiodic and $E^{\nu}\left(X_{0}\right)=0$,
then there exists $C>0$ such that for every $x,y\in\mathbb{Z}$ and
$s\in S$, and initial distribution $\mu,$ 
\[
\sum_{n=0}^{\infty}\left|\PP^{\mu}\left(S_{n}=x,\, X_{n}=s\right)-\PP^{\mu}\left(S_{n}=y,\, X_{n}=s\right)\right|\leq C\left|x-y\right|.
\]
\end{lemma}
\begin{proof}
We denote by $\gamma_{t}\left(\cdot\right):S\rightarrow\mathbb{T}$
the function $e^{it\left(\cdot\right)}$ from the state space $S$
to the circle $\mathbb{T}=\left\{ z\in\mathbb{C}:\left|z\right|=1\right\} $. 

For simplicity of notation we assume that $y=0$. The generalization
to the case when $y\neq0$ is straightforward. Note that since $x\in\mathbb{Z}$,
its enough to prove that the term converges for every $x$ and is
$O(|x|)$ as $|x|\to\infty$. 

By lemma \ref{lem: Expansion of Main Eigenvalue Lemma} and equation
\eqref{eq:Q(t) structure} we may fix $\delta>0$ and positive constants
$C_{1},\, C_{2},C_{3}\,\eta$ such that for $t\in\left(-\delta,\delta\right),$ 

\begin{eqnarray*}
\left|\lambda\left(t\right)\right| & \leq & 1-C_{1}t^{2}\\
\left|\Im\lambda\left(t\right)\right| & \leq & C_{2}\left|t^{3}\right|\\
\left\Vert N\left(t\right)\right\Vert  & \leq & 1-\eta\\
\Pi\left(t\right)\left(\cdot\right) & = & E^{\nu}\left(\cdot\right)+\epsilon\left(t\right)\left(\cdot\right)
\end{eqnarray*}

where $\left\Vert \epsilon\left(t\right)\right\Vert \leq C_{3}\cdot\left|t\right|$.
Set $R_{\delta}=\left[-\pi,\pi\right]\setminus\left(-\delta,\delta\right).$

By the inversion formula for the Fourier Transform,
\begin{eqnarray}
 &  & \sum_{n=0}^{\infty}\left|\PP^{\mu}\left(S_{n}=0,\, X_{n}=s\right)-\PP^{\mu}\left(S_{n}=x,\, X_{n}=s\right)\right|\\
 & = & \sum_{n=0}^{\infty}\left|\Re\left\{ \frac{1}{\left(2\pi\right)^{2}}\int\limits _{-\pi}^{\pi}\intop\limits _{-\pi}^{\pi}E^{\mu}\left(Q\left(t_{1}\right)^{n}\gamma_{t_{2}}\right)\cdot\left(1-e^{-it_{1}x}\right)\cdot e^{-it_{2}s}\, dt_{1}dt_{2}\right\} \right|\label{eq:Fourier Transform Inverse Formula}
\end{eqnarray}

By lemma \ref{lem:Non-Arit. Lemma} there are positive constants $n_{0},$$c,$
$r_{\delta}<1$ such that for every $n\geq n_{0},$ $t\in R_{\delta},$
\[
\left\Vert Q\left(t\right)^{n}\right\Vert \leq cr_{\delta}^{n}.
\]
 Thus,
\begin{eqnarray*}
 &  & \sum_{n=n_{0}}^{\infty}\frac{1}{(2\pi)^{2}}\intop_{-\pi}^{\pi}\int\limits _{R_{\delta}}\left|E^{\mu}\left(Q\left(t_{1}\right)^{n}\gamma_{t_{2}}\right)\cdot\left(1-e^{-it_{1}x}\right)\cdot e^{-it_{2}s}\right|\, dt_{1}dt_{2}\\
 & \leq & \sum_{n=n_{0}}^{\infty}\frac{1}{\left(2\pi\right)^{2}}\int\limits _{-\pi}^{\pi}\intop\limits _{R_{\delta}}cr_{\delta}^{n}\cdot2\, dt_{1}dt_{2}<\infty
\end{eqnarray*}
Here $R_{\delta}=\left[-\pi,\pi\right]\backslash\left(-\delta,\delta\right)$.
It follows that 
\[
\sum\limits _{n=0}^{\infty}\left|\Re\left\{ \frac{1}{\left(2\pi\right)^{2}}\int\limits _{-\pi}^{\pi}\intop\limits _{R_{\delta}}E^{\mu}\left(Q\left(t_{1}\right)^{n}\gamma_{t_{2}}\right)\cdot\left(1-e^{-it_{1}x}\right)\cdot e^{-it_{2}s}\, dt_{1}dt_{2}\right\} \right|<\infty
\]
which is enough for the statement of the lemma. Hence, it remains
to estimate expression \eqref{eq:Fourier Transform Inverse Formula}
where the inner integration is carried over the set $\left(-\delta,\delta\right).$ 

We have, 

\begin{align}
 & \sum_{n=0}^{\infty}\left|\Re\left\{ \frac{1}{\left(2\pi\right)^{2}}\int\limits _{-\pi}^{\pi}\intop\limits _{-\delta}^{\delta}E^{\mu}\left(Q\left(t_{1}\right)^{n}\gamma_{t_{2}}\right)\cdot\left(1-e^{-it_{1}x}\right)\cdot e^{-it_{2}s}\, dt_{1}dt_{2}\right\} \right|\nonumber \\
 & \leq\sum_{n=0}^{\infty}\left|\Re\left\{ \frac{1}{\left(2\pi\right)^{2}}\int\limits _{-\pi}^{\pi}\intop\limits _{-\delta}^{\delta}E^{\mu}\left(\lambda\left(t_{1}\right)^{n}\Pi\left(t_{1}\right)\gamma_{t_{2}}\right)\cdot\left(1-e^{-it_{1}x}\right)\cdot e^{-it_{2}s}\, dt_{1}dt_{2}\right\} \right|\label{eq:Estimate1}\\
 & +\sum_{n=0}^{\infty}\left|\frac{1}{\left(2\pi\right)^{2}}\int\limits _{-\pi}^{\pi}\intop\limits _{-\delta}^{\delta}E^{\mu}\left(N\left(t_{1}\right)^{n}\gamma_{t_{2}}\right)\cdot\left(1-e^{-it_{1}x}\right)\cdot e^{-it_{2}s}\, dt_{1}dt_{2}\right|\nonumber 
\end{align}
Since, $\left\Vert N\left(t\right)\right\Vert \leq1-\eta<1$ and $\left|\gamma_{t}\left(\cdot\right)\right|\equiv1$,
we have 
\begin{equation}
\sum_{n=0}^{\infty}\left|\int\limits _{-\pi}^{\pi}\intop\limits _{-\delta}^{\delta}\frac{1}{\left(2\pi\right)^{2}}E^{\mu}\left(N\left(t_{1}\right)^{n}\gamma_{t_{2}}\right)\cdot\left(1-e^{-it_{1}x}\right)\cdot e^{-it_{2}s}\, dt_{1}dt_{2}\right|\leq\sum_{n=0}^{\infty}C^{'}\cdot\left(1-\eta\right)^{n}\leq C^{'}\frac{1}{\eta}\label{eq:Result1}
\end{equation}
where $C^{'}$ is some constant, which is again enough for our purposes. 

We proceed with estimating \eqref{eq:Estimate1}:

\begin{eqnarray}
 &  & \left|\Re\left\{ \frac{1}{\left(2\pi\right)^{2}}\int\limits _{-\pi}^{\pi}\intop\limits _{-\delta}^{\delta}E^{\mu}\left(\lambda\left(t_{1}\right)^{n}\Pi\left(t_{1}\right)\gamma_{t_{2}}\right)\cdot\left(1-e^{-it_{1}x}\right)\cdot e^{-it_{2}s}\, dt_{1}dt_{2}\right\} \right|\label{eq:Estimate2}\\
 & \leq & \left|\Re\left\{ \frac{1}{\left(2\pi\right)^{2}}\int\limits _{-\pi}^{\pi}\intop\limits _{-\delta}^{\delta}\lambda\left(t_{1}\right)^{n}E^{\nu}\left(\gamma_{t_{2}}\right)\cdot\left(1-e^{-it_{1}x}\right)\cdot e^{-it_{2}s}\, dt_{1}dt_{2}\right\} \right|\nonumber \\
 & + & \left|\Re\left\{ \frac{1}{\left(2\pi\right)^{2}}\int\limits _{-\pi}^{\pi}\intop\limits _{-\delta}^{\delta}E^{\mu}\left(\lambda\left(t_{1}\right)^{n}\cdot\epsilon\left(t_{1}\right)\left(\gamma_{t_{2}}\right)\right)\cdot\left(1-e^{-it_{1}x}\right)\cdot e^{-it_{2}s}\, dt_{1}dt_{2}\right\} \right|.\nonumber 
\end{eqnarray}
Since, $E^{\nu}\left(\gamma_{t_{2}}\right)=E^{\nu}\left(e^{it_{2}X_{0}}\right),$
by using the inversion formula and the Fubini theorem we obtain, 
\begin{eqnarray}
 &  & \left|\Re\left\{ \frac{1}{\left(2\pi\right)^{2}}\int\limits _{-\pi}^{\pi}\intop\limits _{-\delta}^{\delta}\lambda\left(t_{1}\right)^{n}E^{\nu}\left(\gamma_{t_{2}}\right)\cdot\left(1-e^{-it_{1}x}\right)\cdot e^{-it_{2}s}\, dt_{1}dt_{2}\right\} \right|\nonumber \\
 & = & \left|\Re\left\{ \frac{\nu_{s}}{2\pi}\intop\limits _{-\delta}^{\delta}\lambda\left(t\right)^{n}\left(1-e^{-itx}\right)\, dt\right\} \right|\leq\left|\Re\left\{ \frac{1}{2\pi}\intop\limits _{-\delta}^{\delta}\lambda\left(t\right)^{n}\cdot\left(1-e^{-itx}\right)\, dt\right\} \right|\nonumber \\
 & \leq & \frac{1}{2\pi}\intop\limits _{-\delta}^{\delta}\left|\lambda(t)\right|^{n}\cdot\left(1-\cos tx\right)\, dt+\frac{1}{2\pi}\intop\limits _{-\delta}^{\delta}\left|\Im\left\{ \lambda(t)^{n}\right\} \right|\left|\sin tx\right|\, dt.\label{eq:Estimate4}
\end{eqnarray}

By summing over $n$ the first term in expression \eqref{eq:Estimate4}
and by our choice of $\delta,$ 

\begin{eqnarray*}
\sum_{n=0}^{\infty}\frac{1}{2\pi}\intop\limits _{-\delta}^{\delta}\left|\lambda\left(t\right)\right|^{n}\cdot\left(1-\cos tx\right)\, dt & \leq & \frac{1}{2\pi}\intop\limits _{-\delta}^{\delta}\sum_{n=0}^{\infty}\left(1-C_{1}t^{2}\right)^{n}\left(1-\cos tx\right)\, dt\\
 & = & \frac{1}{2\pi}\intop\limits _{-\delta}^{\delta}\frac{1}{C_{1}t^{2}}\left(1-\cos tx\right)\, dt
\end{eqnarray*}

Now, 
\begin{eqnarray*}
\frac{1}{2\pi}\intop\limits _{-\frac{1}{x}}^{\frac{1}{x}}\frac{1}{C_{1}t^{2}}\left(1-\cos tx\right)\, dt & = & \O\left(\left|x\right|\right),\,\left(\left|x\right|\rightarrow\infty\right)
\end{eqnarray*}
and 
\[
\frac{1}{2\pi}\intop\limits _{\frac{1}{x}}^{\delta}\frac{1}{C_{1}t^{2}}\left(1-\cos tx\right)\, dt\leq\frac{1}{\pi}\intop\limits _{\frac{1}{x}}^{\delta}\frac{1}{C_{1}t^{2}}\, dt=\O\left(\left|x\right|\right),\,\left(\left|x\right|\rightarrow\infty\right).
\]

Hence, from the integrand being an even function, we have 
\begin{equation}
\sum_{n=0}^{\infty}\frac{1}{2\pi}\intop\limits _{-\delta}^{\delta}\left|\lambda\left(t\right)\right|^{n}\cdot\left(1-\cos tx\right)\, dt=\O\left(\left|x\right|\right),\,\left(\left|x\right|\rightarrow\infty\right)\label{eq:Result2}
\end{equation}

To estimate the sum over $n$ of the second term in expression \eqref{eq:Estimate4}
we note that, 

\[
\left|\Im\left\{ \lambda\left(t\right)^{n}\right\} \right|\leq\left|\Im\lambda\left(t\right)\right|\cdot n\left|\lambda\left(t\right)\right|^{n-1}\leq C_{2}t^{3}\cdot n\left(1-C_{1}t^{2}\right)^{n-1}
\]
where the first inequality is easily seen to be true by induction
on $n$. Therefore,
\begin{eqnarray}
\sum_{n=0}^{\infty}\frac{1}{2\pi}\intop\limits _{-\delta}^{\delta}\left|\Im\left\{ \lambda\left(t\right)^{n}\right\} \right|\left|\sin tx\right|\, dt & \leq & \frac{1}{2\pi}\intop\limits _{-\delta}^{\delta}\frac{C_{2}\left|t\right|^{3}}{C_{1}t^{4}}\left|\sin tx\right|\, dt\nonumber \\
 & = & \O\left(\left|x\right|\right),\,\left(\left|x\right|\rightarrow\infty\right)\label{eq:Result3}
\end{eqnarray}

This completes the estimation of the sum over $n$ of the first term
in expression \eqref{eq:Estimate2} and all is left is to estimate
the sum over $n$ of the second term in expression \eqref{eq:Estimate2}:

\begin{eqnarray}
 &  & \sum_{n=0}^{\infty}\left|\Re\left\{ \frac{1}{\left(2\pi\right)^{2}}\int\limits _{-\pi}^{\pi}\intop\limits _{-\delta}^{\delta}E^{\mu}\left(\lambda\left(t_{1}\right)^{n}\epsilon_{t_{1}}\left(\gamma_{t_{2}}\right)\right)\cdot\left(1-e^{-it_{1}x}\right)\cdot e^{-it_{2}s}\, dt_{1}dt_{2}\right\} \right|\nonumber \\
 & \leq & \frac{1}{\left(2\pi\right)^{2}}\int\limits _{-\pi}^{\pi}\intop\limits _{-\delta}^{\delta}\sum_{n=0}^{\infty}\left(1-C_{1}t_{1}^{2}\right)^{n}C_{3}\left|t_{1}\right|\cdot\left|1-e^{-it_{1}x}\right|\, dt_{1}dt_{2}\nonumber \\
 & = & \frac{1}{\left(2\pi\right)^{2}}\int\limits _{-\pi}^{\pi}\intop\limits _{-\delta}^{\delta}\frac{C_{3}}{C_{1}\left|t_{1}\right|}\cdot\left|1-e^{-it_{1}x}\right|\, dt_{1}dt_{2}=\O\left(\left|x\right|\right),\,\left(\left|x\right|\rightarrow\infty\right)\label{eq:Result4}
\end{eqnarray}

The lemma follows by combining results \eqref{eq:Result1}, \eqref{eq:Result2},
\eqref{eq:Result3} and \eqref{eq:Result4}.
\end{proof}

\section{\label{sub: proof of the main theorem}Proof of theorem \ref{thm: Main Theorem}.}

Throughout this section we assume that $\left\{ X_{n}\right\} $ satisfies
the assumptions of Theorem \ref{thm: Main Theorem}.

\subsection{A fourth moment inequality for $L_{n}(x)-L_{n}(y)$.}
\begin{lemma}
\label{cor: Fourth moments inequality}There exists $C>0$ such that
for every $x,y\in\ZZ$, 
\[
E^{\mu}\left(\left(L_{n}(x)-L_{n}(y)\right)^{4}\right)\leq C\cdot n\cdot\left|x-y\right|^{2}.
\]
 \end{lemma}
\begin{proof}
Set $Z_{n}=\left(S_{n},X_{n}\right)$. Throughout the proof we use
the fact that for any initial distribution $\mu$ on $S$, $x,y\in\ZZ$,
$n,j\in\NN,(j\leq n$) and $s_{1},s_{2}\in S$, 
\begin{equation}
\PP^{\mu}\left(\left.Z_{n}=\left(x,s_{2}\right)\right|Z_{j}=\left(y,s_{1}\right)\right)=\PP^{s_{1}}\left(Z_{n-j}=\left(x-y,s_{2}\right)\right).\label{eq:FourthMoment3}
\end{equation}

Let $\One_{z}\left(t\right)$ be the indicator function of the set
$\left\{ z\right\} $ and set 
\[
\psi\left(t\right)=\One_{x}\left(t\right)-\One_{y}\left(t\right).
\]

We have

\[
E^{\mu}\left(\left(L_{n}(x)-L_{n}(y)\right)^{4}\right)=\sum_{\bar{i}}E^{\mu}\left(\prod_{j=1}^{4}\psi(S_{i_{j}})\right)
\]
where the summation is carried over all vectors $\bar{i}=\left(i_{1},i_{2},i_{3},i_{4}\right)$
with coordinates equal to integers between $0$ and $n$.

By symmetry, we may assume that $i_{1}\leq i_{2}\leq i_{3}\leq i_{4}$.
For a fixed $\bar{i}$

\begin{equation}
E^{\mu}\left(\prod_{j=1}^{4}\psi(S_{i_{j}})\right)=\sum_{\xi}\left(-1\right)^{\sigma}E^{\mu}\left(\One_{\xi_{1}}\left(S_{i_{1}}\right)\psi\left(S_{i_{2}}\right)\One_{\xi_{2}}\left(S_{i_{3}}\right)\psi\left(S_{i_{4}}\right)\right)\label{eq: FourthMoment2}
\end{equation}
where the sum is carried out over all vectors $\bar{\xi}=\left(\xi_{1},\xi_{2}\right)$
with coordinates equal to either $x$ or $y$ and $\mbox{ \ensuremath{\sigma}=\ensuremath{\sigma\left(\bar{\xi}\right)} }$
is the number of coordinates that equal $y$.

Fixing $\bar{i}$ and $\bar{\xi}$ , and using \eqref{eq:FourthMoment3}
we obtain 

\begin{eqnarray}
E^{\mu}\left(\One_{\xi_{1}}\left(S_{i_{1}}\right)\psi\left(S_{i_{2}}\right)\One_{\xi_{2}}\left(S_{i_{3}}\right)\psi\left(S_{i_{4}}\right)\right) & =\sum_{\bar{s}\in S^{4}} & \left[\PP^{\mu}\left(Z_{i_{1}}=\left(\xi_{1},s_{1}\right)\right)\cdot\right.\nonumber \\
 &  & \left.F\left(i_{1},i_{2},i_{3}\right)\cdot\left(\PP^{s_{3}}\left(S_{i_{4}}=x\right)-\PP^{s_{3}}\left(S_{i_{4}}=y\right)\right)\right]\label{eq: FourthMoment1}
\end{eqnarray}
where 

\begin{eqnarray*}
F\left(i_{1},i_{2},i_{3}\right) & = & \left[\PP^{s_{1}}\left\{ Z_{i_{2}-i_{1}}=\left(x-\xi_{1},s_{2}\right)\right\} \PP^{s_{2}}\left\{ Z_{i_{3}-i_{2}}=\left(\xi_{2}-x,s_{3}\right)\right\} \right.\\
 &  & \left.-\PP^{s_{1}}\left\{ Z_{i_{2}-i_{1}}=\left(y-\xi_{1},s_{2}\right)\right\} \PP^{s_{2}}\left\{ Z_{i_{3}-i_{2}}=\left(\xi_{2}-y,s_{3}\right)\right\} \right].
\end{eqnarray*}

At this point we take absolute values in \eqref{eq: FourthMoment1}
and sum over all vectors $\bar{i}$ obtaining 
\begin{eqnarray*}
\left|\sum_{\bar{i}}E^{\mu}\left(\One_{\xi_{1}}\left(S_{i_{1}}\right)\psi\left(S_{i_{2}}\right)\One_{\xi_{2}}\left(S_{i_{3}}\right)\psi\left(S_{i_{4}}\right)\right)\right| & \leq\sum_{\bar{s}\in S^{4}} & \sum_{\bar{i}}\left|\PP^{\mu}\left(Z_{i_{1}}=\left(\xi_{1},s_{1}\right)\right)\cdot\right.\\
 &  & \left.F\left(i_{1},i_{2},i_{3}\right)\cdot\left(\PP^{s_{3}}\left(S_{i_{4}}=x\right)-\PP^{s_{3}}\left(S_{i_{4}}=y\right)\right)\right|\\
 & \leq\sum_{\bar{s}\in S^{4}} & \sum_{0\leq i_{1}\leq i_{2}\leq i_{3}\leq n}\left[C\cdot\frac{1}{\sqrt{i_{1}+1}}\cdot\right.\\
 &  & \left.\left|F\left(i_{1},i_{2},i_{3}\right)\right|\cdot\left|x-y\right|\right]
\end{eqnarray*}
where $C$ is some constant. Here, Lemma \ref{lem:LLT} and Lemma
\ref{pro:Potential Kernel Estimate} were used to obtain the last
inequality and $i_{1}+1$ appears instead of $i_{1}$, so that the
expression will be defined for $i_{1}=0$.

Note that, 
\begin{eqnarray*}
|F(i_{1},i_{2},i_{3})| & \leq & \PP^{s_{1}}\left(Z_{i_{2}-i_{1}}=\left(x-\xi_{1},s_{2}\right)\right)\left|\PP^{s_{2}}\left\{ Z_{i_{3}-i_{2}}=\left(\xi_{2}-x,s_{3}\right)\right\} -\PP^{s_{2}}\left\{ Z_{i_{3}-i_{2}}=\left(\xi_{2}-y,s_{3}\right)\right\} \right|\\
 & + & \PP^{s_{2}}\left\{ Z_{i_{3}-i_{2}}=\left(\xi_{2}-y,s_{3}\right)\right\} \left|\PP^{s_{1}}\left\{ Z_{i_{2}-i_{1}}=\left(x-\xi_{1},s_{2}\right)\right\} -\PP^{s_{1}}\left\{ Z_{i_{2}-i_{1}}=\left(y-\xi_{1},s_{2}\right)\right\} \right|
\end{eqnarray*}

Substituting $|F(i_{1},i_{2},i_{3})|$ by the above expression, using
Lemma \ref{lem:LLT} to deduce that sums of the type 
\[
\sum\limits _{k=0}^{n}P^{\mu}\left(Z_{n}=\left(x,s\right)\right)
\]
 are bounded by constant times $\sqrt{n}$, and applying Lemma \ref{pro:Potential Kernel Estimate}
we conclude that 
\[
\left|\sum_{\bar{i}}E\left(\One_{\xi_{1}}\left(S_{i_{1}}\right)\psi\left(S_{i_{2}}\right)\One_{\xi_{2}}\left(S_{i_{3}}\right)\psi\left(S_{i_{4}}\right)\right)\right|\leq C\cdot n\left|x-y\right|^{2}.
\]

Since there is only a finite number of possible vectors $\xi$ in
\eqref{eq: FourthMoment2} (there are exactly $4$) the claim follows. \end{proof}
\begin{corollary}
\textup{\label{cor:Tightness of Local Times} For every $x,y\in\RR$
\[
\PP^{\mu}\left(\left|t_{n}(x)-t_{n}(y)\right|>\epsilon\right)\leq\frac{C}{\epsilon^{4}}\left|x-y\right|^{2}.
\]
}\end{corollary}
\begin{proof}
By Lemma \ref{cor: Fourth moments inequality}, 
\begin{eqnarray*}
E^{\mu}\left(\left(t_{n}(x)-t_{n}(y)\right)^{4}\right) & = & \frac{1}{n^{2}}E^{\mu}\left[\left(L_{n}\left(\left\lfloor \sqrt{n}x\right\rfloor \right)-L_{n}\left(\left\lfloor \sqrt{n}y\right\rfloor \right)\right)^{4}\right]\\
 & \leq & C\left|x-y\right|^{2}.
\end{eqnarray*}

The claim now follows by Chebychev's inequality.
\end{proof}

\subsection{Relative compactness of $t_{n}\left(x\right)$ in $D$.}

A sequence $\left\{ X_{n}\right\} $ of random variables taking values
in a standard Borel Space $\left(X,\BB\right)$ is called tight if
for every $\epsilon>0$ there exists a compact $K\subset X$ such
that for every $n\in\NN$,
\[
P_{n}(K)>1-\epsilon,
\]
where $P_{n}$ denotes the distribution of $X_{n}$ . By Prokhorov's
Theorem relative compactness of $t_{n}(x)$ in $D$ is equivalent
to tightness. Therefore we are interested in characterizing tightness
in $D$.

For $x\left(t\right)$ in $D_{\left[-m,m\right]},$$T\subseteq\left[-m,m\right]$
set 
\[
\omega_{x}\left(T\right)=\sup_{s,t\in T}\left|x(s)-x(t)\right|
\]
and 

\[
\omega_{x}(\delta):=\sup_{\left|s-t\right|<\delta}\left|x\left(s\right)-x\left(t\right)\right|.
\]

$\omega_{x}\left(\delta\right)$ is called the modulus of continuity
of $x$. Due to the Arzela - Ascoli theorem it plays a central role
in characterizing tightness in the space $C\left[-m,m\right]$ of
continuous functions on $\left[-m,m\right]$, with a Borel $\sigma$-algebra
generated by the topology of uniform convergence. 

The function that plays in $D\left[-m,m\right]$ the role that the
modulus of continuity plays in $C\left[-m,m\right]$ is defined by
\[
\omega_{x}'(\delta)=\inf_{\left\{ t_{i}\right\} }\max_{1\leq i\leq v}\omega(\left[t_{i},t_{i+1})\right),
\]
where $\left\{ t_{i}\right\} $ denotes a $\delta$ sparse partition
of $[-m,m]$, i.e. $\left\{ t_{i}\right\} $ is a partition $-m=t_{1}<t_{2}<...<t_{v+1}=m$
such that ${\displaystyle \min_{1\leq i\leq v}\left|t_{i+1}-t_{i}\right|>\delta}$.
It is easy to check that if $\frac{1}{2}>\delta>0$, and $m\geq1,$
\[
\omega'_{x}(\delta)\leq\omega_{x}\left(2\delta\right).
\]
For details see \cite[Sections 12 and 13]{Bil}. The next theorem
is a characterization of tightness in the space $D$. 
\begin{theorem}
\cite[Lemma 3, p.173]{Bil} (1)The sequence $t_{n}$ is tight in $D$
if and only if its restriction to $\left[-m,m\right]$ is tight in
$D_{\left[-m,m\right]}$ for every $m\in\RR_{+}$.

(2) The sequence $t_{n}$ is tight in $D_{\left[-m,m\right]}$ if
and only if the following two conditions hold:

(i) $\forall x\in\left[-m,m\right],\;\lim\limits _{a\rightarrow\infty}\limsup\limits _{n\rightarrow\infty}P\left[\left|t{}_{n}\left(x\right)\right|\geq a\right]=0.$

(ii) $\forall\epsilon>0,\,\lim\limits _{\delta\rightarrow0}\limsup\limits _{n\rightarrow\infty}P\left[\omega_{t_{n}}^{'}\left(\delta\right)\geq\epsilon\right]=0.$ \end{theorem}
\begin{remark}
See \cite[Thm. 13.2]{Bil} and Corollary after wards. \label{rem: Norm Boundness}Conditions
$\left(i\right)$ and $\left(ii\right)$ of the previous theorem imply
that 
\[
\lim\limits _{a\rightarrow\infty}\limsup\limits _{n\rightarrow\infty}P\left[\sup_{x\in[-m,m]}\left|t{}_{n}\left(x\right)\right|\geq a\right]=0.
\]
\end{remark}
\begin{proposition}
\label{pro:The-sequence is tight}The sequence $\left\{ t_{n}(x)\right\} _{n=1}^{\infty}$
is tight. \end{proposition}
\begin{proof}
We prove that condition $2(i)$ holds. 

Fix $\epsilon>0,$ $x\in\RR.$ Since the Brownian Motion $W_{\sigma}(t)$
satisfies 
\[
\lim_{M\to\infty}\PP\left(\sup_{t\in[0,1]}\left|W_{\sigma}(t)\right|>M\right)=0
\]
and $W_{n}\left(t\right)$ converges in distribution to $W_{\sigma}\left(t\right)$
there are $M,n_{0}$ such that for all $n>n_{0},$

\[
\PP^{\mu}\left(\sup_{t\in[0,1]}\left|W_{n}(t)\right|>M\right)<\epsilon.
\]

By definition of $t_{n}\left(x\right),$ it follows that if $\left|x\right|>M,$
$n>n_{0},$ 
\[
\PP^{\mu}\left(\left|t_{n}\left(x\right)\right|>0\right)<\epsilon.
\]

Now, if $\left|x\right|\leq M$, by corollary \ref{cor:Tightness of Local Times},

\begin{eqnarray*}
\PP^{\mu}\left(\left|t_{n}\left(x\right)\right|>a\right) & \leq & \PP^{\mu}\left(\left|t_{n}\left(M+1\right)\right|>0\right)+\PP^{\mu}\left(\left|t_{n}\left(x\right)-t_{n}\left(M+1\right)\right|>a\right)\\
 & \leq & \epsilon+\frac{4C\cdot\left(M+1\right)^{2}}{a^{4}}
\end{eqnarray*}

and the last expression can be made less then $2\epsilon$ for sufficiently
large $a$. 

To prove condition $2(ii)$ WLOG we may assume that $m\geq1$. Since
$\omega'_{x}(\delta)\leq\omega_{x}\left(2\delta\right)$, it is sufficient
to prove that the stronger condition
\begin{equation}
\forall\epsilon>0.\ \lim\limits _{\delta\rightarrow0}\limsup\limits _{n\rightarrow\infty}\PP^{\mu}\left[\sup_{x,y\in[-m.m];\left|x-y\right|<\delta}\left|t_{n}(x)-t_{n}(y)\right|\geq\epsilon\right]=0\label{eq: sufficient condition for tightness}
\end{equation}
holds. 

Let $\epsilon>0$. By Corollary \ref{cor:Tightness of Local Times}
there exists $C>0$ such that for all $x,y:\ \frac{1}{\sqrt{n}}\leq|x-y|\leq1$
\begin{equation}
\PP^{\mu}\left(\left|t_{n}(x)-t_{n}(y)\right|>\epsilon\right)\leq\frac{C}{\epsilon^{4}}\left|x-y\right|^{2}.\label{eq:sufficient condition for tightness 2}
\end{equation}
Let $\delta>0$ and $n>\delta^{-2}$, notice that $t_{n}(x)$ is constant
on segments of the form $\left[\frac{j}{\sqrt{n}},\frac{j+1}{\sqrt{n}}\right)$,
hence 
\[
\PP^{\mu}\left(\sup_{x,y\in[-m.m];\left|x-y\right|<\delta}\left|t_{n}(x)-t_{n}(y)\right|\geq4\epsilon\right)\leq\sum\limits _{|k\delta|\leq m}\PP^{\mu}\left(\sup\limits _{k\delta\sqrt{n}\leq j\leq\left(k+1\right)\delta\sqrt{n}}\left|t_{n}\left(k\delta\right)-t_{n}\left(\frac{j}{\sqrt{n}}\right)\right|\geq\epsilon\right).
\]
By \cite[Theorem 10.2]{Bil} it follows from \eqref{eq:sufficient condition for tightness 2}
that there exists $C_{2}>0$ such that 
\begin{equation}
\PP^{\mu}\left(\sup\limits _{k\delta\sqrt{n}\leq j\leq\left(k+1\right)\delta\sqrt{n}}\left|t_{n}\left(k\delta\right)-t_{n}\left(\frac{j}{\sqrt{n}}\right)\right|\geq\epsilon\right)\leq\frac{C_{2}}{\epsilon^{4}}\delta^{2}.\label{eq:Billingsley2}
\end{equation}

Therefore, 
\[
\PP^{\mu}\mbox{\ensuremath{\left(\sup_{x,y\in[-m.m];\left|x-y\right|<\delta}\left|t_{n}(x)-t_{n}(y)\right|\geq4\epsilon\right)\leq\frac{2C_{2}m}{\epsilon^{4}}\delta\xrightarrow[\delta\to0]{}}0. }
\]

\end{proof}

\subsection{\label{sub:Identifying the Only Possible Limit Point}Identifying
$t_{\sigma}\left(x\right)$ as the Only Possible Limit of a Subsequence
of $\left\{ t_{n}\left(x\right)\right\} _{n\in\NN}$.}
\begin{proposition}
\label{pro:Identification ot the limit}Assume that the sequence $\left\{ X_{n}\right\} $
satisfies the assumptions of Theorem \ref{thm: Main Theorem}. Let
$t_{n_{k}}\left(x\right)$ be some subsequence of $t_{n}\left(x\right)$
that converges in distribution to some limit $q\left(x\right).$ Then
$q\left(x\right)\deq t_{\sigma}\left(x\right)$. \end{proposition}
\begin{proof}
Let $q$ be a limit point of $\left\{ t_{n}(x)\right\} $. i.e. there
exists $n_{k}\to\infty$ such that $t_{n_{k}}\overset{d}{\Rightarrow}q(x).$ 

By \cite[Section 2]{KS}, for every $a,b\in\RR,$ the function on
$D_{\left[0,1\right]}$ defined by 
\[
f\rightarrow\int_{0}^{1}\One_{[a,b)}\left(f\right)dt
\]
 is continuous in the Skorokhod topology. 

Therefore, since $W_{n}(t)\overset{d}{\Rightarrow}W_{\sigma}(t)$
, for every $a,b$ the random variables 

\[
\Lambda_{n}([a,b))=\int_{0}^{1}\One_{[a,b)}\left(W_{n}(t)\right)dt
\]
converge in distribution to the occupation measure of the Brownian
motion defined by 
\[
\Lambda([a,b))=\int_{0}^{1}\One_{[a,b)}\left(W_{\sigma}(t)\right)dt.
\]

The local time $t_{\sigma}$ of the Brownian motion is an almost surely
continuous function of $x$ such that for fixed $a,b\in\RR$ the equality
\begin{equation}
\int_{a}^{b}t_{\sigma}(x)dx=\int_{0}^{1}\One_{[a,b)}\left(W_{\sigma}(t)\right)dt\label{eq:B.motion local time and occupation}
\end{equation}
holds almost surely (see \cite{MP}). 

By straightforward calculations using definitions, we have 
\begin{equation}
\left|\int_{0}^{1}\One_{[a,b)}\left(W_{n}(t)\right)dt-\int_{a}^{b}t_{n}(x)dx\right|\leq\int\limits _{\frac{\left\lfloor \sqrt{n}a\right\rfloor }{\sqrt{n}}}^{\frac{\left\lfloor \sqrt{n}a\right\rfloor +1}{\sqrt{n}}}t_{n}\left(x\right)dx+\int\limits _{\frac{\left\lfloor \sqrt{n}b\right\rfloor }{\sqrt{n}}}^{\frac{\left\lfloor \sqrt{n}b\right\rfloor +1}{\sqrt{n}}}t_{n}\left(x\right)dx.\label{eq:Difference in densities}
\end{equation}

Now, 
\begin{eqnarray*}
\PP^{\mu}\left(\left|\int\limits _{\frac{\left\lfloor \sqrt{n}a\right\rfloor }{\sqrt{n}}}^{\frac{\left\lfloor \sqrt{n}a\right\rfloor +1}{\sqrt{n}}}t_{n}\left(x\right)dx\right|>\epsilon\right) & \leq & \PP^{\mu}\left(\sup_{x\in\left[a-1,b+1\right]}\left|t_{n}\left(x\right)\right|>M\right)\\
 & + & \PP^{\mu}\left(\left|\int\limits _{\frac{\left\lfloor \sqrt{n}a\right\rfloor }{\sqrt{n}}}^{\frac{\left\lfloor \sqrt{n}a\right\rfloor +1}{\sqrt{n}}}t_{n}\left(x\right)dx\right|>\epsilon,\sup_{x\in\left[a-1,b+1\right]}\left|t_{n}\left(x\right)\right|\leq M\right).
\end{eqnarray*}

The second summand on the right side of the above inequality tends
to $0$ since the integral is less than $\frac{M}{\sqrt{n}}$ .The
first summand is arbitrarily close to $0$ for $M,n$ large enough,
by Remark \ref{rem: Norm Boundness}. Same reasoning applied to both
summands of equation \eqref{eq:Difference in densities} gives 
\[
\left|\int_{0}^{1}\One_{[a,b)}\left(W_{n}(t)\right)dt-\int_{a}^{b}t_{n}(x)dx\right|\overset{d}{\Rightarrow}0.
\]

It follows that the distributional limits of $\int_{0}^{1}\One_{[a,b)}\left(W_{n}(t)\right)dt$
and $\int_{a}^{b}t_{n}(x)dx$ coincide (see \cite[Thm.3.1]{Bil}).
Therefore, for every $a,b\in\RR$, 
\[
\int_{a}^{b}t_{n}(x)dx\overset{d}{\Rightarrow}\int_{a}^{b}t_{\sigma}(x)dx.
\]
For every $a,b$ the function 
\[
f\mapsto g_{a,b}(f)=\int_{a}^{b}f(x)dx
\]
 is a continuous function from $D$ to $\RR$. Therefore, since $t_{n_{k}}\overset{d}{\Rightarrow}q(x)$
, it follows that 
\[
\int_{a}^{b}t_{n_{k}}(x)dx\overset{d}{\Rightarrow}\int_{a}^{b}q(x)dx.
\]
Hence, 
\[
\int_{a}^{b}q(x)dx\overset{d}{=}\int_{a}^{b}t_{\sigma}(x)dx.
\]

Since the collection of functions $g_{a,b},$$a,b\in\RR$ generates
$\mathcal{G}(D)$, it follows that 

\[
q\deq t_{\sigma}.
\]

\end{proof}

\section{\label{sec:Borodin's-Theorem.}Proof of Borodin's Theorem (Theorem
\ref{thm: Borodin for MC}).}

We start by forming the mutual probability space on which, $W'_{n}$
and $W'_{\sigma}$ are defined. 

Let $\left\{ t_{i}\right\} $ be a dense subset of $\left[0,\infty\right]$
and for each $n$ let us regard the infinite vector 
\[
\xi_{n}:=\left(W_{n},\, l_{n}\left(t_{1},\cdot\right),\, l_{n}\left(t_{2},\cdot\right),\,...\right)
\]
 as a random variable taking values in the space $\Pi:=D_{+}\times D^{\NN}$ 

Since, a countable product of Polish Spaces is Polish, to prove convergence
of the sequence $\left\{ \xi_{n}\right\} $ we need to establish tightness
and uniquely identify the limit. By Tichonov's theorem, tightness
in each coordinate of the vectors $\xi_{n}$ separately (established
in Theorem \ref{thm: Main Theorem}) implies tightness for the whole
sequence $\left\{ \xi_{n}\right\} $. 

To identify the limit, generalizing the method in the proof of proposition
\ref{pro:Identification ot the limit} to $\Pi$, for each sequence
$\alpha=\left\{ (a_{i},b_{i})\right\} _{i\in\NN}$ we define a continuous
function $g_{\alpha}$ on $\Pi$ (taking values in $D_{+}\times\RR^{\NN}$)
by 
\begin{eqnarray*}
\left(f,h_{1},h_{2},...\right) & \overset{g_{\alpha}}{\rightarrow} & \left(f,\,\int\limits _{a_{1}}^{b_{1}}h_{1},\,\int\limits _{a_{2}}^{b_{2}}h_{2},...\right).
\end{eqnarray*}

The functions $\left\{ g_{\alpha}\right\} $ where $\alpha$ goes
over all possible sequences of intervals in $\RR^{2}$ generate the
Borel $\sigma$-field of $\Pi.$ Now, the arguments in the proof of
proposition \eqref{pro:Identification ot the limit}may be easily
modified to show that 
\[
g_{\alpha}\left(\xi_{n}\right)\overset{d}{\rightarrow}g_{\alpha}\left(\xi\right)
\]
 where $\xi:=\left(W'_{\sigma},\, l\left(t_{1},\cdot\right),\, l\left(t_{2},\cdot\right),\,...\right).$
This uniquely identifies the limit, establishing the convergence in
distribution of the sequence $\xi_{n}$ to $\xi$. Now, by Skorokhod's
theorem (See \cite[Theorem 6.7]{Bil}), we may define processes $\xi'_{n}$
,$\xi':\Omega'\rightarrow$ on some probability space $\left(\Omega',\BB',\PP\right)$
such that $\xi'_{n}$ and $\xi'$ have the same distribution as $\xi_{n}$
and $\xi$ respectively, and such that almost surely, $\xi'_{n}$
converges to $\xi'$. Note that for $n\in\NN$, by left continuity,
for almost every $\omega\in\Omega'$, the vector 
\[
\xi'_{n}=\left(W'_{\sigma},l'_{n}\left(t_{1},x\right),l'_{n}\left(t_{2},x\right),...\right)
\]
 uniquely determines a function $l'_{n}\left[t,x\right):\left[0,1\right]\times\left(-\infty,\infty\right)\rightarrow\RR,$
which coincides with the coordinates of $\xi'_{n}$ for $t\in\left\{ t_{i}\right\} $,
and is Cadlag in each variable. 

It is clear that the functions $l'_{n}\left(t,x\right)$ are a.s equal
to the local time of $W'_{n}$. Similarly, the vector $\xi'$ uniquely,
determines a function $l'\left(t,x\right)$ which a.s equals the Brownian
local time of $W'_{\sigma},$ and is therefore continuous. The a.s.
convergence of $\xi'_{n}$ to $\xi'$ implies a.s. convergence in
the $J_{1}$ topology of $W'_{n}$ to $W'_{\sigma}$, thus proving
\textit{$\left({\rm 2}\right)$ }of theorem \ref{thm: Borodin for MC}. 

It remains to prove that $l'_{n},\, l',\, n=1,2,...$ satisfy the
statement in \textit{$\left({\rm 3}\right)$}. To see this fix $T>0.$
For $h>0,$ whose value is to be determined later, we extract from
$\left\{ t_{i}\right\} $ a finite partition of $\left[0,T\right]$
with mesh less than $h$ and denote it by 
\[
T:\,0=t_{0}<t_{1}<...<t_{l}=1.
\]
Now for $i=1,..,l$, the process $l'_{n}\left(t_{i},\cdot\right)$,
a.s converges to $l'\left(t_{i},\cdot\right)$ in the metric of the
space $D$. Since, the limit is a continuous function, by properties
of the $J_{1}$ topology, this convergence must be uniform on compact
subsets of $\RR$. We may conclude that for every compact subset $K\subseteq\RR$
, every $\epsilon>0$, and almost all $\omega\in\Omega'$, there exists
$N_{0}$, such that for every $n>N_{0}$, 
\begin{equation}
\sup_{x\in K,t\in T}\left|l'_{n}\left(t,x\right)-l'\left(t,x\right)\right|<\epsilon.\label{eq: Uniform Approx.}
\end{equation}

Now, by monotonicity of the local time as a function of time, if $l'_{n}\left(t,x\right)-l'\left(t,x\right)\geq0$
then
\[
l'_{n}\left(t,x\right)-l'\left(t,x\right)\leq l'_{n}\left(t_{i+1},x\right)-l'\left(t_{i},x\right)
\]
and if $l'_{n}\left(t,x\right)-l'\left(t,x\right)\leq0$ then 
\[
l'\left(t,x\right)-l'_{n}\left(t,x\right)\leq l'\left(t_{i+1},x\right)-l'_{n}\left(t_{i},x\right)
\]
where $t_{i},\, t_{i+1}$ are points in $T$ satisfying 
\[
t_{i}\leq t\leq t_{i+1}.
\]

Hence, 
\begin{eqnarray}
\PP\left(\sup_{t\in\left[0,1\right],x\in\RR}\left|l'_{n}\left(t,x\right)-l'\left(t,x\right)\right|>2\epsilon\right) & \leq & \PP\left(\sup_{t\in\left[0,1\right]}\left|W'_{n}\right|>M\right)+\PP\left(\sup_{t\in\left[0,1\right]}\left|W'_{\sigma}\right|>M\right)\label{eq:Loc.Time.1}\\
 &  & +\PP\left(\sup_{i,\, x\in\left[-M,M\right]}\left|l'_{n}\left(t_{i+1},x\right)-l'\left(t_{i},x\right)\right|>\epsilon\right)\label{eq:Loc.Time.2}\\
 &  & +\PP\left(\sup_{i,\, x\in\left[-M,M\right]}\left|l'_{n}\left(t_{i},x\right)-l'\left(t_{i+1},x\right)\right|>\epsilon\right)\label{eq:Loc.Time.3}
\end{eqnarray}
where $i=1,...l.$

As in the proof of Proposition \eqref{pro:The-sequence is tight}
expression \eqref{eq:Loc.Time.1} is small for sufficiently large
$M$. Expressions \eqref{eq:Loc.Time.2} and \eqref{eq:Loc.Time.3}
are handled similarly. We have
\begin{eqnarray*}
\PP\left(\sup_{i,\, x\in\left[-M,M\right]}\left|l'_{n}\left(t_{i+1},x\right)-l'\left(t_{i},x\right)\right|>\epsilon\right) & \leq & \PP\left(\sup_{t\in T,\, x\in\left[-M,M\right]}\left|l'_{n}\left(t_{i+1},x\right)-l'\left(t_{i+1},x\right)\right|>\frac{\epsilon}{2}\right)\\
 &  & +\PP\left(\sup_{i=0,1...l,\, x}\left|l'\left(t_{i+1},x\right)-l'\left(t_{i},x\right)\right|>\frac{\epsilon}{2}\right).
\end{eqnarray*}

The first expression on the right equals $0$ for all $n$ sufficiently
large, by \ref{eq: Uniform Approx.}. The second expression, by a.s.
continuity of Brownian local time, can be made arbitrarily small by
taking $h$ to be sufficiently small. This concludes the proof.

\section{\label{sec:The-Periodic-Case}The periodic Case }

Let $X_{1},X_{2},....$ be an i.i.d. sequence in the domain of attraction
of the Gaussian distribution. In this setting Borodin proved convergence
of local times under the condition of what we call non-arithmeticity
of the random walk $S_{n}$, that is 
\[
E\left(e^{itX_{i}}\right)=1\ \Leftrightarrow\ t\in2\pi\ZZ.
\]

Our notion of strong aperiodicity in the i.i.d case is equivalent
to the stronger condition that

\[
\left|E\left(e^{itX_{i}}\right)\right|=1\ \Leftrightarrow\ t\in2\pi\ZZ.
\]

In this section we show that the methods of this paper are sufficient
to replace the assumption of strong aperiodicity by a weaker non-arithmeticity
in the i.i.d case and for Markov chains that are almost onto (see
section \ref{sub:The-periodic-case}).

We assume that $E\left(X_{1}\right)=0$ and that 
\[
\Sigma:=\left\{ x|\PP\left(S_{n}=x\right)>0\,\, for\,\, some\,\, n\right\} =\ZZ.
\]
Otherwise, since$\Sigma$ is a subgroup of $\ZZ$, in the recurrent
case, we'll need to relabel the state space. 

Let 
\[
p=\inf\left\{ k\in\NN:\ \PP\left(S_{k}=0\right)>0\right\} 
\]
 be the period of the random walk. It follows that 
\[
\left|E\left(e^{itX}\right)\right|=1\ \Leftrightarrow t\in\frac{2\pi}{p}\ZZ
\]

In this case the random variable $X_{1}$ takes values in a proper
coset of $\ZZ$. The periodic structure of $S_{n}$ is that for $n\in\NN$
and $k\in\left\{ 1,...,p-1\right\} $, $S_{np+k}$ takes values in
one of the $p-1$ cosets of $\ZZ$ of the form $p\ZZ+j$, $j=0,...,p-1.$ 

In this case the potential kernel estimate( Lemma \ref{pro:Potential Kernel Estimate})
should be corrected to state that for every $x,y\in\ZZ$, 
\begin{equation}
\sum_{n=0}^{\infty}\left|\sum_{k=0}^{p-1}\PP\left(S_{np+k}=x\right)-\PP\left(S_{np+k}=y\right)\right|\leq C\left|x-y\right|,\label{eq: JP periodic}
\end{equation}
which is a special case of the corollary from lemma 7 in \cite{JP}.
Lemma \ref{lem:LLT} also holds in this case (see Lemma 2 in \cite{JP}).
These two ingredients are enough to prove the fourth moment inequality,
Lemma \ref{cor: Fourth moments inequality}. To see how this is done
notice that (here we keep the notation of Lemma \ref{cor: Fourth moments inequality})
\begin{eqnarray}
E\left(\One_{\xi_{1}}\left(S_{i_{1}}\right)\psi\left(S_{i_{2}}\right)\One_{\xi_{2}}\left(S_{i_{3}}\right)\psi\left(S_{i_{4}}\right)\right) & = & \left[\PP\left(S_{i_{1}}=\xi_{1}\right)\cdot F\left(i_{1},i_{2},i_{3}\right)\cdot\right.\label{eq: sumation over xi}\\
 &  & \left.\left(\PP\left(S_{i_{4}-i_{3}}=x\right)-\PP\left(S_{i_{4}-i_{3}}=y\right)\right)\right]\nonumber 
\end{eqnarray}
where 
\begin{eqnarray}
F\left(i_{1},i_{2},i_{3}\right) & = & \left[\PP\left(S_{i_{2}-i_{1}}=x-\xi_{1}\right)\PP\left(S_{i_{3}-i_{2}}=\xi_{2}-x\right)\right.\label{eq: F for random walks}\\
 &  & \left.-\PP\left(S_{i_{2}-i_{1}}=y-\xi_{1}\right)\PP\left(S_{i_{3}-i_{2}}=\xi_{2}-y\right)\right]\\
 & = & \PP\left(S_{i_{2}-i_{1}}=x-\xi_{1}\right)\left[\PP\left(S_{i_{3}-i_{2}}=\xi_{2}-x\right)-\PP\left(S_{i_{3}-i_{2}}=\xi_{2}-y\right)\right]\nonumber \\
 &  & +\PP\left(S_{i_{3}-i_{2}}=\xi_{2}-y\right)\left[\PP\left(S_{i_{2}-i_{1}}=x-\xi_{1}\right)-\PP\left(S_{i_{2}-i_{1}}=y-\xi_{1}\right)\right]\nonumber 
\end{eqnarray}
Therefore, since 
\begin{eqnarray*}
E^{\mu}\left(L_{n}(x)-L_{n}(y)\right)^{4} & = & \sum_{\overrightarrow{\xi}}\sum_{\bar{i}}E^{\mu}\left(\One_{\xi_{1}}\left(S_{i_{1}}\right)\psi\left(S_{i_{2}}\right)\One_{\xi_{2}}\left(S_{i_{3}}\right)\psi\left(S_{i_{4}}\right)\right),
\end{eqnarray*}
from \eqref{eq: sumation over xi} and \eqref{eq: F for random walks}
it follows that we need to estimate two sums. The first is 
\begin{eqnarray}
\sum_{i_{1}=1}^{n}\left[\PP\left(S_{i_{1}}=\xi_{1}\right)\sum_{i_{2}=0}^{n-i_{1}}\PP\left(S_{i_{2}}=x-\xi_{1}\right)\sum_{i_{3}=0}^{n-i_{1}-i_{2}}\left\{ \PP\left(S_{i_{3}}=\xi_{2}-x\right)-\PP\left(S_{i_{3}}=\xi_{2}-y\right)\right\} \right.\label{eq: long ugly summation for periodic}\\
\left.\sum_{i_{4}=0}^{n-i_{1}-i_{2}-i_{3}}\left\{ \PP\left(S_{i_{4}}=x\right)-\PP\left(S_{i_{4}}=y\right)\right\} \right].\nonumber 
\end{eqnarray}
In order to bound this term we first notice that 
\[
\sum_{i_{4}=0}^{n-i_{1}-i_{2}-i_{3}}\left\{ \PP\left(S_{i_{4}}=x\right)-\PP\left(S_{i_{4}}=y\right)\right\} =\sum_{l=0}^{\left\lfloor \left(n-i_{1}-i_{2}-i_{3}\right)/p\right\rfloor }\sum_{j=0}^{p-1}\left\{ \PP\left(S_{l\cdot p+j}=x\right)-\PP\left(S_{l\cdot p+j}=y\right)\right\} \pm p.
\]
here $a=b\pm c$ means $\left|a-b\right|\leq c$. We do a similar
rearrangement to the sum of $i_{3}$. It then follows that the term
in \eqref{eq: long ugly summation for periodic} is strictly smaller
than 
\begin{eqnarray*}
\sum_{i_{1}=1}^{n}\left[\PP\left(S_{i_{1}}=\xi_{1}\right)\sum_{i_{2}=0}^{n-i_{1}}\PP\left(S_{i_{2}}=x-\xi_{1}\right)\cdot\left(p+\sum_{l_{1}=0}^{\infty}\left|\sum_{j_{1}=0}^{p-1}\left\{ \PP\left(S_{i_{3}}=\xi_{2}-x\right)-\PP\left(S_{i_{3}}=\xi_{2}-y\right)\right\} \right|\right)\right.\\
\left.\cdot\left(p+\sum_{l=0}^{\infty}\left|\sum_{j=0}^{p-1}\left\{ \PP\left(S_{l\cdot p+j}=x\right)-\PP\left(S_{l\cdot p+j}=y\right)\right\} \right]\right)\right].
\end{eqnarray*}
The conclusion follows from a double application of the inequality
\eqref{eq: JP periodic} and the bound in Lemma 2 of \cite{JP}. The
second term is dealt with similarly and the rest of the proof of Borodin's
theorem remains the same. Thus we have:
\begin{theorem}
Let $X_{1},X_{2},...,X_{n},..$ be an i.i.d sequence with $E\left(X_{1}\right)=0$
and $E\left(X_{1}^{2}\right)=1$ such that the random walk $S_{n}$
is recurrent. Then there exists a probability space $\left(X',\BB',\PP'\right)$
and processes $W'_{n},\, W'_{\sigma}:X\rightarrow D_{+}$ such that:\end{theorem}
\begin{enumerate}
\item \textit{$W_{n}\overset{d}{=}W'_{n};\, W\overset{d}{=}W'$.}
\item \textit{With probability one $W'_{n}$ converges to $W'$ uniformly
on compact subsets of $\left[0,\infty\right)$.}
\item \textit{For every $\epsilon,T>0$ the processes $l'_{n}\left(t,x\right)$
and $l'\left(t,x\right)$ defined with respect to $W'_{n}$ and $W'_{\sigma}$
satisfy the relationship:
\[
\lim_{n\rightarrow\infty}\PP'\left(\sup_{t\in\left[0,T\right],\, x\in\RR}\left|l'_{n}\left(t,x\right)-l'\left(t,x\right)\right|>\epsilon\right)=0.
\]
}
\end{enumerate}

\subsection{\label{sub:The-periodic-case}The periodic case of finite state Markov
chains }

To drop the assumption of strong aperiodicity when $X_{1},..,X_{n},...$
is a Markov chain we use the dynamical setting introduced in section
1. Let $P:S\times S\to[0,1]$ be the transition matrix, $\mu$ be
an initial distribution on $S$ and $P^{\mu}$ the probability measure
on $S^{\mathbb{N}}$ generated by $\mu$ and $P$. We then look at
the Markov shift $\left(S^{\mathbb{N}},\BB.P^{\mu},T\right)$ where
$T$ is the shift and $\BB$ is the Borel $\sigma$-algebra generated
by the cylinder sets 
\[
\left[s_{0}s_{1}...s_{n}\right]=\left\{ \omega\in S^{\NN}:\,\omega_{i}=s_{i},\, i=0,1,...,n\right\} .
\]

In this case $X_{n}=f\left(T^{n}\omega\right)=\omega_{n}$ and if
the function $f$ is periodic there exists a solution to 
\[
e^{itf\left(\omega\right)}=\lambda\frac{\varphi\left(\omega\right)}{\varphi\left(T\omega\right)}
\]
where either $t\notin2\pi\ZZ$ or $\varphi$ is not constant. To make
the same analysis as for random walks we need to exclude the case
$\varphi\neq const$ (then summation over the Markov Chain will yield
the same periodic structure as in the i.i.d case). A sufficient condition
for that is that the system $\left(X,\BB,P_{\mu},T,\alpha\right)$
is almost onto with respect to the partition $\alpha=\left\{ \left[s_{0}\right]:\ s_{0}\in S\right\} $,
meaning that for every $a,b\in\alpha$ there exist sets $a_{0},...,a_{n}$
such that $a_{0}=b$, $a_{n}=c$ and $Ta_{k}\cap Ta_{k+1}\neq\emptyset$
(for details see \cite[Section 3]{AD}).

\section{\label{sec:Applications-to-complexity}Applications to complexity
of random walks in random sceneries with a Markov chain base}

A Random Walk in Random Scenery is a skew product probability preserving
transformation which is defined as follows:

The \textit{random scenery} is an invertible probability preserving
transformation $\left(Y,\mathcal{C},\nu,S\right)$ and the \textit{random
walk in random scenery $\left(Y,\mathcal{C},\nu,S\right)$ with jump
random variable $\xi$ }(assumed $\ZZ$-valued) is the skew product
$\left(Z,\BB(Z),m,T\right)$ defined by 
\[
Z:=\Omega\times Y,\ m:=\mu_{\xi}\times\nu,\ T(w,y)=\left(Rw,S^{w_{0}}y\right)
\]
where 
\[
\left(\Omega,\BB\left(\Omega\right),\mu_{\xi},R\right)=\left(\ZZ^{\ZZ},\BB\left(\ZZ^{\ZZ}\right),\prod{\rm dist}\xi,{\rm shift}\right)
\]
 is the shift of the independent jump random variables. 

Aaronson \cite{Aar}, assuming $\xi$ is in the domain of attraction
of an $\alpha-$stable law with $\alpha>0$, has studied the relative
complexity and the relative entropy dimension of $T$ over the base.
We refer the reader to \cite{Aar} for the definitions. 

One can ask if Aaronson's result \cite[Theorem 3, Section 1]{Aar}
remain true if we change the base $\left(\Omega,\BB\left(\Omega\right),\mu_{\xi},R\right)$
to a strongly aperiodic, irreducible, finite state Markov Shift $\left(\Omega,\BB\left(\Omega\right),P^{\mu},R\right)$. 

With Theorem \ref{thm: Borodin for MC} at hand, Aaronson's proof,
which can even be simplified since the Brownian motion is a.s continuous,
can be carried out verbatim.

\section{Appendix}
\begin{proposition}
\label{Trivial proposition}Let $\left(X,\mathcal{B},\mu\right)$
be a probability space, $f:X\rightarrow\mathbb{S}^{1}$(the unit sphere)
measurable. Then 
\[
\left|\int f\left(x\right)\, d\mu\right|=1
\]
implies $f=const$. \\
\end{proposition}
\begin{proof}
Assume that $f$ is not constant. Then there are two disjoint intervals
$I_{1},I_{2}\subseteq\mathbb{T}$ such that $\mu\left(f^{-1}\left(I_{1}\right)\right)>0$
and $\mu\left(f^{-1}\left(I_{1}\right)\right)>0$. It follows that
\[
\left|\int_{X}f\left(x\right)\, d\mu\right|\leq\left|\int_{f^{-1}(I_{1})}f(x)d\mu+\int_{f^{-1}(I_{2})}f(x)d\mu\right|+\mu\left(X\backslash\left(f^{-1}(I_{1})\cup f^{-1}(I_{2})\right)\right)<1.
\]
This is true since if for $j\in\{1,2\}$
\[
\left|\frac{1}{\mu\left(f^{-1}(I_{j})\right)}\int_{f^{-1}(I_{j})}f(x)d\mu\right|<1
\]
then it is trivial, else 
\[
\frac{1}{\mu\left(f^{-1}(I_{j})\right)}\int_{f^{-1}(I_{j})}f(x)d\mu\in I_{j},\ j=1,2
\]
 and the mean of two different values on the unit sphere is of modulus
strictly less than $1$ .\end{proof}

\end{document}